\documentclass{amsart}
\usepackage{graphicx}
\usepackage[all]{xy}
\usepackage[ansinew]{inputenc}
\usepackage{amsmath}
\usepackage{amscd}
\usepackage{amssymb}
\usepackage{latexsym}
\usepackage[active]{srcltx}
\usepackage{mathrsfs}
\usepackage{color}

\newtheorem{thm}{Theorem}

\theoremstyle{definition}

\newtheorem{lem}[thm]{Lemma}
\newtheorem{prop}[thm]{Proposition}
\newtheorem{defn}[thm]{Definition}

\newtheorem{fact}[thm]{Fact}

\newtheorem*{thmA}{Theorem A}

\numberwithin{equation}{section}

\newcommand{\R}{\mathbb{R}}

\newcommand{\K}{\mathbb{K}}

\def \<{\langle}
\def \>{\rangle}

\def \((  {(\!(}
\def \)) {)\!)}

\begin{document}

\title[]{An analogue of the Baire Category Theorem}%
\author{Philipp Hieronymi}
\address{University of Illinois at Urbana-Champaign\\
Department of Mathematics\\
1409 W. Green Street\\
Urbana, IL 61801\\
USA}
\email{P@hieronymi.de}


\date{\today}

\thanks{Some version of this paper will appear in the Journal of Symbolic Logic.}

\maketitle

\begin{abstract}
Every definably complete expansion of an ordered field satisfies an analogue of the Baire Category Theorem.
\end{abstract}

\section{Introduction}

Let $\K$ be an expansion of an ordered field $(K,<,+,\cdot)$. We say $\K$ is \emph{definably complete} if every bounded subset of $K$ definable in $\K$ has a supremum in $K$. Such structures were first studied by Miller in \cite{ivp}. The main result of this paper is the following definable analogue of the Baire Category Theorem.

\begin{thmA}\label{mainthm} Let $\K$ be definably complete. Then there exists \emph{no} set $Y \subseteq K_{>0}\times K$ definable in $\K$ such that
\begin{itemize}
\item[(i)] $Y_t$ is nowhere dense for $t\in K_{>0}$,
\item[(ii)]$Y_s \subseteq Y_t$ for $s,t\in K_{>0}$ with $s<t$, and
\item[(iii)] $\bigcup_{t \in K_{>0}}Y_t = K$,
\end{itemize}
where $Y_t$ denotes the set $\{a \in K : (t,a) \in Y\}$.
\end{thmA}
Theorem A is a positive answer to a conjecture of Fornasiero and Servi raised in \cite{fornasiero,fs}. By their work, Theorem A implies that definable versions of many standard facts from real analysis hold in $\K$. Among these are a definable analogue of the Kuratowski-Ulam Theorem, a restricted version of Sard's Theorem and several results in the model theoretic study of Pfaffian functions (see \cite{fs,fs2,fs3} for details). \\

A short remark about the proof of Theorem A is in order. A definably complete structure does not need to be complete in the topological sense. For this reason the strategy of the classical proof of the Baire Category Theorem to define a sequence of real numbers by recursion is not viable in our setting, as such a sequence might not converge. However, by \cite{fornasiero} (see Fact 3 below) it is enough to consider a definable complete $\K$ that defines a closed and discrete set which is mapped by a definable function onto a dense subset of $K$. In such a situation techniques are available that are based on the idea of definable approximation schemes first used by the author in \cite{discrete}. These ideas allow us to replace the use of recursion in the classical proof by an explicit definition of an appropriate sequence.

\subsection*{Notation} In the rest of the paper $\K$ will always be a fixed definably complete expansion of an ordered field $K$. When we say a set is definable in $\K$, we always mean definable with parameters from $\K$. We will use $a,b,c$ for elements of $K$. The letters $d,e$ will always denote elements of a discrete set $D$. Given a subset $X$ of $K^{n}\times K^m$ and $a \in K^n$, we denote the set $\{b : (a,b) \in X\}$ by $X_a$. We write $\overline{X}$ for the topological closure of $X$ in the usual order topology.

\section{Facts about definably complete fields}

\noindent In this section we recall several facts about definably complete expansions of ordered fields. For more details and background, see \cite{ivp}. Let $\K$ be a definably complete expansion of an ordered field.

\begin{fact}\label{unbounded} Let $Y \subseteq K$ be a non-empty closed set definable in $\K$. Then $Y$ contains a minimum and a maximum iff $Y$ is bounded.
\end{fact}

\begin{fact}[$\textrm{\cite[Lemma 1.9]{ivp}}$]\label{closedunion} Let $Y\subseteq K^2$ be definable in $\K$ such that $Y_a$ is closed and bounded and $Y_a \supseteq Y_b\neq \emptyset$ for every $a,b \in K$ with $a<b$. Then $\bigcap_{a\in K} Y_a\neq \emptyset$.
\end{fact}

\noindent We say that $\K$ is \emph{definably Baire} if it satisfies the conclusion of Theorem A.  The proof of Theorem A uses the following result of Fornasiero as a starting point.

\begin{fact}[$\textrm{\cite[Corollary 6.6]{fornasiero}}$] \label{forn} If $\K$ is \emph{not} definably Baire, then there exists an unbounded, closed and discrete set $D\subseteq K_{\geq 0}$ and a function $f:D \to K$ such that $f$ is definable in $\K$ and $f(D)$ is dense in $K$.
\end{fact}

\noindent The strategy for the proof of Theorem A is to establish the following statement: A definably complete expansion of an order field that defines an unbounded, closed and discrete set which is mapped by a definable function onto a dense set, is definably Baire. Note that there are many instances where we already know Theorem A holds. Since $\R$ is a Baire space, every expansion of the real field is definably Baire. Moreover, any o-minimal expansion of an ordered field is definably Baire. For more examples in this direction and related results for expansions of ordered groups, see Dolich, Miller and Steinhorn \cite[3.5]{dms}.

\section{Proof of Theorem A}
Let $\K$ be a definably complete expansion of an ordered field $(K,<,+,\cdot)$. Towards a contradiction, we assume that there is an increasing family $(Y_t)_{t\in K_{>0}}$ of definable nowhere dense sets such that $K = \bigcup_{t \in K_{>0}} Y_t$. Set $Y_0:=\emptyset$. Define $X_t:=K\setminus Y_t$. Then $X_t$ is dense in $K$. By replacing $Y_t$ by its topological closure, we can assume that $X_t$ is open.\\
By Fact \ref{forn} there is also an unbounded, closed and discrete set $D\subseteq K_{\geq 0}$ definable in $\K$ and a map $f: D \to K$ definable in $\K$ such that the image of $D$ under $f$ is dense in $K$.  Since $D$ is cofinal in $K_{>0}$,
$$
K = \bigcup_{d\in D} Y_d.
$$
\noindent Let $\beta: K \to D\cup \{0\}$ be the function that maps $c$ to the largest $d\in D \cup \{0\}$ such that $c\in X_d$. Note that $\beta$ is unbounded. Further let $\gamma : K \to [0,1]$ map $c\in K$ to the supremum of the set of elements $b$ in $(0,1)$ such that $\big(c-2b,c\big)\subseteq X_{\beta(c)}$. Since $X_{\beta(c)}$ is open, $\gamma(c)>0$. We will write $I_c$ for the open interval $\big(c-\gamma(c),c\big)$. We will use the following properties of $\beta$ and $\gamma$.
\begin{lem}\label{lemma4} Let $c \in K$. Then
$$
\emptyset \neq \overline{I_c} \subseteq X_{\beta(c)}.
$$
\end{lem}

\begin{defn} Let $c\in K$. Define $S_c\subseteq D$ as the set
$$
\Big\{ d \in D \ : \ f(d)>c \wedge \forall e \in D \ e<d\rightarrow \big(f(e)<c \vee f(d)<f(e)\big)\Big\}.
$$
Moreover, let $S_c^{\beta}\subseteq D$ be
$$
\Big\{ d \in S_c \ : \ \forall e \in D \ (e< d \wedge e \in S_c) \rightarrow \beta(f(e))<\beta(f(d))\Big\}.
$$
\end{defn}
\noindent The elements of the set $S_c$ can be interpreted as the set of best approximations of $c$ from the right. Compare this to the approximation arguments used in \cite{discrete}. Note that $S_c$ is always unbounded, because it does not contain a maximum. The set $S_c^{\beta}$ is always non-empty for every $c \in K$, since it contains the minimum of $S_c$. But a priori there is no reason why $S_c^{\beta}$ should be unbounded. In fact, it might even be finite for some $c \in K$. The advantage of $S_c^{\beta}$ over $S_c$ is that the composition $\beta \circ f$ is strictly increasing on $S_c^{\beta}$.

\begin{defn} Let $\delta: K \to D$ be the function that maps $c$ to the largest $d \in S_{c}^{\beta}$ such that for all $e_1,e_2\in S_c^{\beta}$ with $e_1<e_2\leq d$
$$
f(e_2) \in I_{f(e_1)}.
$$
Define $J_c \subseteq K$ by
$$J_c:=\bigcap_{e \in S_c^{\beta} \cap [0,\delta(c)]} I_{f(e)}.$$
\end{defn}

\begin{lem}\label{welldef} The function $\delta$ is well-defined.
\end{lem}
\begin{proof} Let $c\in K$. Towards a contradiction, suppose that $\delta(c)$ is not defined. Then $S_{c}^{\beta}$ is unbounded by Fact \ref{unbounded} and for all $e_1,e_2\in S_c^{\beta}$ with $e_1<e_2$
$$
f(e_2) \in I_{f(e_1)}.
$$
Then for every $e \in S_c^{\beta}$ the set
$$
\bigcap_{e_1 \in S_{c}^{\beta}, e_1 \leq e} \overline{I_{f(e_1)}} \subseteq X_{\beta(f(e))}
$$
contains $f(e)$, and hence is non-empty and closed. By Fact \ref{closedunion} and Lemma \ref{lemma4}
$$
\emptyset \neq \bigcap_{e \in S_{c}^{\beta}} \overline{I_{f(e)}} \subseteq \bigcap_{e \in S_{c}^{\beta}} X_{\beta(f(e))}.
$$
Since $S_c^{\beta}$ is unbounded and $\beta\circ f$ is strictly increasing on $S_c^{\beta}$, the set $\{ \beta(f(e)) : e \in S_{c}^{\beta}\}$ does not contain a maximum. Thus by Fact \ref{unbounded} it is unbounded. Hence it is cofinal in $D$ and
$$
\bigcap_{e \in S_{c}^{\beta}} X_{\beta(f(e))}= \bigcap_{d\in D} X_d.
$$
This is a contradiction, since $\bigcap_{d\in D} X_d$ is empty.
\end{proof}

\begin{lem}\label{nonemptyopen} Let $c\in K$. Then $J_c$ is a non-empty open interval.
\end{lem}
\begin{proof} Let $d_1 \in D$ be the largest element of $S_c^{\beta} \cap [0,\delta(c)]$ such that
$$\bigcap_{e \in S_c^{\beta} \cap [0,d_1]} I_{f(e)}$$
is a non-empty open interval. Such an element exists, since $S_c^{\beta}$ is non-empty and $I_a$ is a non-empty open interval for every $a \in K$. Towards a contradiction, suppose that $d_1 < \delta(c)$. Let $d_2 \in D$ be the smallest element in $S_c^{\beta} \cap [0,\delta(c)]$ larger than $d_1$. Since $d_2 \leq \delta(c)$, $f(d_2) \in  I_{f(e)}$ for all $e \in S_c^{\beta}$ with $e<d_2$. Hence
$$I_{f(d_2)} \cap \bigcap_{e \in S_c^{\beta} \cap [0,d_1]} I_{f(e)}$$
is a non-empty open interval. This is a contradiction to the maximality of $d_1$. Hence $d_1=\delta(c)$.
\end{proof}

\noindent Note that for every $c\in K$
$$
f(\delta(c)) \in \overline{J_c} \subseteq X_{\beta(f(\delta(c)))}.
$$
 In order to find a counter-example to the statement $\bigcap_{d\in D} X_d = \emptyset$, we will amalgamate sets of the form $S_c^{\beta}\cap [0,\delta(c)]$. For this purpose we introduce the following notion of an extension.

\begin{defn} For $c_1,c_2 \in K$, we say that $c_2$ \emph{extends} $c_1$ if $\delta(c_1) < \delta(c_2)$ and
$$
S_{c_1}^{\beta} \cap \big[0,\delta(c_1)\big] = S_{c_2}^{\beta}\cap \big[0,\delta(c_1)\big].
$$
\end{defn}

\noindent In the following we will construct an unbounded definable subset $E_0$ of $D$ such that for all $d,e \in E_0$ with $d<e$, $f(e)$ extends $f(d)$. Given such a set $E_0$, we will be able create a contradiction as in the proof of Lemma \ref{welldef} (see proof of Theorem A below). With that goal in mind, we start by establishing several properties of extensions. First note that being an extension is transitive. If $c_3$ extends $c_2$ and $c_2$ extends $c_1$, then $c_3$ extends $c_1$.

\begin{lem}\label{remark1} Let $c_1,c_2 \in K$ be such that $c_2$ extends $c_1$. Then
\begin{itemize}
\item[(i)] $J_{c_2} \subseteq J_{c_1}$, and
\item[(ii)] $\beta(f(\delta(c_2))) > \beta(f(\delta(c_1)))$.
\end{itemize}
\end{lem}
\begin{proof} (i) Since $c_2$ extends $c_1$, $S_{c_1}^{\beta} \cap \big[0,\delta(c_1)\big]\subseteq S_{c_2}^{\beta}\cap \big[0,\delta(c_2)\big]$. Hence $J_{c_2} \subseteq J_{c_1}$.\newline
(ii) Since $c_2$ extends $c_1$, $\delta(c_1) \in S_{c_2}^{\beta}$. Since $\beta\circ f$ is strictly increasing on $S_{c_2}^{\beta}$, $\beta(f(\delta(c_2))) > \beta(f(\delta(c_1))).$
\end{proof}

\begin{lem}\label{maxleft} Let $c\in K$ and $d \in D$. If the set
$$L:= \{ f(e) : e \in D, e<d, f(e)<c\}$$
is non-empty, then it contains a maximum.
\end{lem}
\begin{proof} Suppose $L$ is non-empty. Then the set
$$
\Big\{ e_1 \in D \ : \ f(e_1)<c \wedge \forall e_2 \in D (e_2<e_1)\rightarrow \big(f(e_2)>c \vee f(e_2)<f(e_1)\big)\Big\} \cap (0,d)
$$
is bounded and non-empty. Thus it contains a maximum, say $e_3$. By the definition, the image of $e_3$ under $f$ is the maximum of $L$.
\end{proof}

\begin{prop}\label{onestep} Let $c \in K$. Then there exists $d\in D$ such that $f(d)$ extends $c$.
\end{prop}
\begin{proof} By Lemma \ref{nonemptyopen}, the set
$$
A:= K_{>c} \cap J_c
$$
is a non-empty open interval. We will construct $d,d_1 \in D$ such that $f(d_1)\in A$, $f(d)$ extends $c$ and $d_1$ is the smallest element in $S_{f(d)}^{\beta}$ larger than $\delta(c)$. Because $f(d_1)\in A$, $d_1$ witnesses that $\delta(f(d))>\delta(c)$.\\

\noindent Since $f(D)$ is dense in $K$, we can define $d_1 \in D$ as the smallest element in $D$ such that $f(d_1) \in A$ and $\beta(f(d_1)) > \beta(f(\delta(c))$. Since $c<f(d_1)<f(\delta(c))$ and $\delta(c) \in S_c$, we have $d_1>\delta(c)$. We now choose $d$. By Lemma \ref{maxleft}, the set
$$
\{ f(e) : e \in D, e<d_1, f(e)<f(d_1)\}
$$
has a maximum, say $f(d_2)$ for some $d_2 \in D$. By density of $f(D)$, we can choose $d\in D$ such that $f(d)\in A\cap \big(f(d_2),f(d_1)\big)$.\\

\noindent Since $c<f(d)<f(\delta(c))$,
$$
S_c^{\beta} \cap [0,\delta(c)] = S_{f(d)}^{\beta} \cap [0,\delta(c)].
$$
It is only left to establish that $\delta(f(d))> \delta(c)$. Since $f(d_1) \in A$, it is enough to show that $d_1$ is the smallest element in $S_{f(d)}^{\beta}$ larger than $\delta(c)$. By the choice of $d$, we have that for all $e \in D$ with $e<d_1$
\begin{equation}\label{eq1}
c< f(d) < f(d_1) < f(e) \textrm{ or } f(e) < f(d).
\end{equation}
Hence $d_1 \in S_{f(d)}$.
Let $e \in D$ be such that $\delta(c)< e < d_1$. We will show that $e\notin S_{f(d)}^{\beta}$. This then directly implies that $d_1 \in S_{f(d)}^{\beta}$ and that $d_1$ is the smallest such element larger than $\delta(c)$. By minimality of $d_1$ either $f(e)\notin A$ or $\beta(f(e)) \leq \beta(f(\delta(c))$. In both cases we have to check that $e\notin S_{f(d)}^{\beta}$. If $\beta(f(e)) \leq \beta(f(\delta(c))$, then $e\notin S_{f(d)}^{\beta}$ because $\delta(c) \in S_{f(d)}^{\beta}$ and $\beta \circ f$ is strictly increasing on $S_{f(d)}^{\beta}$. Now consider the case that $f(e) \notin A$. Since $A$ is an interval and $f(\delta(c))\in \overline{A}$, either $f(e)<f(d_1)$ or $f(e)\geq f(\delta(c))$. If $f(e)\geq f(\delta(c))$, then  $e\notin S_{f(d)}^{\beta}$ because $\delta(c) \in S_{f(d)}^{\beta}$. If $f(e) < f(d_1)$, then $f(e)<f(d)$ by \eqref{eq1}. Hence $e\notin S_{f(d)}^{\beta}$.  Hence  $d_1$ is the smallest element in $S_{f(d)}^{\beta}$ larger than $\delta(c)$.
\end{proof}

\begin{defn} Define $E$ as the set of $e \in D$ such that there is no $d \in D$ with $d<e$ and
$$
S_{f(e)}^{\beta} \cap \big[0,\delta(f(e))\big] = S_{f(d)}^{\beta}\cap \big[0,\delta(f(e))\big].
$$
\end{defn}

\noindent The set $E$ is defined in a way to make sure that if $e \in E$, $d\in D$ and $f(d)$ extends $f(e)$, then $e<d$.

\begin{lem}\label{extends} Let $c\in K$. The set
$
\{ e \in E  :  f(e) \textrm{ extends } c\}
$
is unbounded.
\end{lem}
\begin{proof} Let $d_1\in D$ be the smallest element in $D$ such that $f(d_1)$ extends $c$. It is easy to see that $d_1 \in E$. Towards a contradiction, suppose there exists $e\in E$ such that $e$ is the largest element in $E$ such that $f(e)$ extends $c$. By Proposition \ref{onestep}, let $d\in D$ be the smallest element of $D$ such that $f(d)$ extends $f(e)$. Because $e\in E$, $d>e$. Since $f(e)$ extends $c$, so does $f(d)$. Moreover, since $e \in E$ and $d$ is the smallest element in $D$ such that $f(d)$ extends $f(e)$, $d$ is in $E$ as well. This contradicts the maximality of $e$. Hence the set $\{ e \in E  :  f(e) \textrm{ extends } c\}$ is unbounded.
\end{proof}

\begin{lem}\label{preextends} Let $d_1,d_2,d_3 \in D$ be such that $d_2 \in E$ and $d_1<d_2<d_3$. If $f(d_3)$ extends $f(d_1)$ and $f(d_3)$ extends $f(d_2)$, then $f(d_2)$ extends $f(d_1)$.
\end{lem}
\begin{proof} Towards a contradiction, suppose that $\delta(f(d_2)) \leq \delta((f(d_1))$.
Since $f(d_3)$ extends both $f(d_1)$ and $f(d_2)$,
$$S_{f(d_2)}^{\beta} \cap \big[0,\delta(f(d_2))\big] = S_{f(d_1)}^{\beta}\cap \big[0,\delta(f(d_2))\big].$$
This contradicts $d_2 \in E$.  Hence $\delta(f(d_2)) > \delta(f(d_1))$ and
$$S_{f(d_1)}^{\beta} \cap \big[0,\delta(f(d_1))\big] = S_{f(d_2)}^{\beta}\cap \big[0,\delta(f(d_1))\big].$$
\end{proof}
\begin{defn}
Let $d_0$ be the smallest element in $E$.
Define $E_0\subseteq E$ as the set of elements $d$ of $E$ satisfying the following two properties:
\begin{itemize}
\item either $f(d)$ extends $f(d_0)$ or $d=d_0$,
\item if there are $e_1,e_2 \in E$ such that $d_0\leq e_1< d$, $f(d)$ extends $f(e_1)$ and $e_2$ is the smallest element in $E$ larger than $e_1$ such that $f(e_2)$ extends $f(e_1)$, then either $d=e_2$ or $f(d)$ extends $f(e_2)$.
\end{itemize}
\end{defn}

\noindent The set $E_0$ is definable in $\mathbb{K}$, since both $E$ and the property of being an extension are definable in $\K$.

\begin{lem}\label{preextends2} Let $d \in E_0$. If $e$ is the smallest element in $E$ larger than $d$ such that $f(e)$ extends $f(d)$, then $e \in E_0$.
\end{lem}
\begin{proof} Since $f(e)$ extends $f(d)$, $f(e)$ extends $f(d_0)$. Let $e_1,e_2 \in E$ be such that $d_0\leq e_1<e$, $f(e)$ extends $f(e_1)$ and $e_2$ is the smallest element in $E$ larger than $e_1$ such that $f(e_2)$ extends $f(e_1)$. If $e_1=d$, we get $e_2=e$ by minimality of $e$. If $e_1 < d$, then $f(d)$ extends $f(e_1)$ by Lemma \ref{preextends}. Since $d \in E_0$, either $d=e_2$ or $f(d)$ extends $f(e_2)$. Thus $f(e)$ extends $f(e_2)$. Hence we can reduce to the case that $e_1> d$. Since $f(e)$ extends both $f(d)$ and $f(e_1)$, $f(e_1)$ extends $f(d)$ by Lemma \ref{preextends}. But then $e_1=e$ by the minimality of $e$. Hence $e \in E_0$.
\end{proof}

\begin{prop}\label{extends2} Let $d,e \in E_0$. If $d<e$, then $f(e)$ extends $f(d)$.
\end{prop}
\begin{proof} Consider the set
$$
Z:=\{ d \in E_0 : \forall e_1,e_2 \in E_0 (e_1 \leq d \wedge e_1 <e_2) \rightarrow (f(e_2) \textrm{ extends } f(e_1))\}.
$$
It is enough to show that $Z$ is unbounded. Since $d_0 \in Z$ by definition of $E_0$, $Z$ is non-empty. For a contradiction, suppose $d_1 \in D$ is the largest element in $Z$. Let $d_2$ be the smallest element in $E$ such that $f(d_2)$ extends $f(d_1)$. By Lemma \ref{preextends2}, $d_2 \in E_0$. For every $e\in E_0$ with $e > d_1$, either $e=d_2$ or $f(e)$ extends $f(d_2)$ by definition of $E_0$. Hence $d_2 \in Z$. This contradicts the maximality of $d_1$.
\end{proof}

\begin{proof}[Proof of Theorem A] We will show that
$$
\emptyset \neq \bigcap_{d \in E_0} \  \overline{J_{f(d)}} \subseteq \bigcap_{d\in D} X_d.
$$
This contradicts the assumption that the family $(Y_d)_{d \in D}$ witnesses that $\K$ is not definably Baire, and hence establishes Theorem A.\\
By Proposition \ref{extends2} and Lemma \ref{remark1}, we have for all $d_1,d_2 \in E_0$ with $d_1<d_2$
$$
\overline{J_{f(d_2)}} \subseteq \overline{J_{f(d_1)}}.
$$
By Fact \ref{closedunion},
$$\emptyset \neq \bigcap_{d \in E_0} \  \overline{J_{f(d)}} \subseteq \bigcap_{d\in E_0} X_{\beta(f(\delta(f(d)))}.$$
By Lemma \ref{extends} and \ref{preextends2}, the set $E_0$ has no maximum and hence is unbounded by Fact \ref{unbounded}. Hence $\{ \beta(f(\delta(f(d)))) : d \in E_0\}$ is unbounded as well by Proposition \ref{extends2} and Lemma \ref{remark1}. Thus the set $\bigcap_{d\in D} X_d$ is equal to $\bigcap_{d \in E_0} X_{\beta(f(\delta(f(d))))}$ and in particular non-empty. \end{proof}

 

\end{document}